\documentclass{mathscan} 

\usepackage[utf8]{inputenc} 

\usepackage{lipsum}

\usepackage{geometry} 
\usepackage{graphicx} 

\usepackage[english]{babel}

\usepackage{booktabs} 
\usepackage{array} 
\usepackage{paralist} 
\usepackage{verbatim} 
\usepackage{subfig} 

\usepackage{caption}

\usepackage{amsfonts,amsmath,amssymb,bbm,amsthm,amsbsy}

\usepackage{mathtools}

\makeatletter							
\newtheorem*{rep@theorem}{\rep@title}
\newcommand{\newreptheorem}[2]{%
\newenvironment{rep#1}[1]{%
 \def\rep@title{#2 \ref{##1}}%
 \begin{rep@theorem}}%
 {\end{rep@theorem}}}
\makeatother

\newtheorem{thm}{Theorem}[section]

\newtheorem{prop}[thm]{Proposition}

\newtheorem{claim}[thm]{Claim}

\newtheorem*{thm*}{Theorem}
\newtheorem*{lemma*}{Lemma}
\newtheorem*{prop*}{Proposition}
\newtheorem*{corr*}{Corrolary}
\newtheorem*{claim*}{Claim}

\newtheorem{theorem}{Theorem}

\theoremstyle{remark}
\newtheorem{rmk}[thm]{Remark}

\newtheorem*{rmk*}{Remark}
\newtheorem*{conj*}{Conjecture}
\newtheorem*{quest*}{Question}

\theoremstyle{definition}
\newtheorem{defn}[thm]{Definition}
\newtheorem{exmp}[thm]{Example}

\newtheorem*{defn*}{Definition}
\newtheorem*{exmp*}{Example}

\newreptheorem{theorem}{Theorem}
\newreptheorem{corollary}{Corollary}
\newreptheorem{proposition}{Proposition}

\usepackage{fancyhdr} 
\usepackage{tikz}
\usepackage{tikz-cd}

\usepackage{IEEEtrantools}

\newenvironment{equ*}[1]{\begin{IEEEeqnarray*}{#1}}{\end{IEEEeqnarray*}}

\newcommand{\R}{\mathbb{R}}

\newcommand{\Z}{\mathbb{Z}}

\newcommand{\Ecl}{\mathcal{E}}

\newcommand{\Hrm}{\mathrm{H}}

\newcommand{\col}{\colon}

\DeclareMathOperator{\Heis}{\mathcal{H}_3}

\DeclareMathOperator{\erm}{\rm{e}}
\DeclareMathOperator{\crm}{\rm{c}}
\DeclareMathOperator{\im}{\rm{im}}
\DeclareMathOperator{\Erm}{\rm{E}}
\DeclareMathOperator{\drm}{\rm{d}}
\DeclareMathOperator{\zrm}{\rm{z}}
\DeclareMathOperator{\Zrm}{\rm{Z}}
\DeclareMathOperator{\orm}{\rm{o}}
\DeclareMathOperator{\Aut}{\rm{Aut}}
\DeclareMathOperator{\Out}{\rm{Out}}
\DeclareMathOperator{\Inn}{\rm{Inn}}

\author{Nicolaus Heuer}
\date{\today}
\title{Low-Dimensional Bounded Cohomology and Extensions of Groups}
\address{Department of Mathematics\\
  University of Oxford}
\email[N.~Heuer]{heuer@maths.ox.ac.uk}

\begin{document}

\begin{abstract}
Bounded cohomology of groups was first studied by Gromov in 1982 in his seminal paper \cite{gromov}. 
Since then it has sparked much research in Geometric Group Theory.
However, it is notoriously hard to explicitly compute bounded cohomology, even for most basic ``non-positively curved'' groups.
On the other hand, there is a well-known interpretation of \emph{ordinary} group cohomology in dimension $2$ and $3$ in terms of group extensions.
The aim of this paper is to make this interpretation available for \emph{bounded} group cohomology.
This will involve \emph{quasihomomorphisms} as defined and studied by Fujiwara--Kapovich \cite{fk}.
\end{abstract}

\maketitle

\section{Introduction}

In \cite{gromov} Gromov studied bounded cohomology of groups in connection to minimal volume of manifolds.
Since then bounded cohomology has been established as an independent active research field due to its connection to other areas in Geometric Group Theory.
Most prominent applications include stable commutator length (\cite{calegari}), circle actions (\cite{ghys_fr}, \cite{ghys}, \cite{circle-actions}) and the Chern Conjecture. See \cite{monod_invitation} and \cite{frigerio} for an introduction to the topic.

For a group $G$ and a normed $G$-module $V$, denote by $\Hrm_b^n(G,V)$ the $n$-dimensional bounded cohomology of $G$ with coefficients in $V$; see Subsection \ref{subsec:bounded cohomology}.
$\Hrm_b^n(G,V)$ is notoriously hard to compute explicitly.
Consider the most most basic case of $V = \R$ with a trivial $G$ action.
If $G$ is amenable then it is known that $\Hrm_b^n(G,\R)=0$ for all $n \geq 1$.
On the other hand if $G$ is ``non-positively curved'' then $\Hrm_b^2(G,\R)$ and $\Hrm_b^3(G,\R)$  are typically infinite dimensional as an $\R$-vectorspace, for example for acylindrically hyperbolic groups; see \cite{hull_osin} and
\cite{fps}. 
However, there is no full characterisation of all bounded classes in $\Hrm^n_b(G, \R)$ for $n=2,3$.
For $n \geq 4$, $\Hrm_b^n(G,\R)$ is usually fully unknown, even if $G$ is a non-abelian free group. 

On the other hand, for \emph{ordinary} $n$-dimensional group cohomology $\Hrm^n(G,V)$ there is a well-known characterisation for $n=2,3$ in terms of
 \emph{group extensions}. The aim of this paper is to make this well-known correspondence available for \emph{bounded} cohomology. For this, we first recall the classical connection between group extensions and ordinary group cohomology.

\begin{defn} \label{defn:equivalence of group extensions}
An \emph{extension} of a group $G$ by a group $N$ is a short exact sequence of groups
\begin{align} \label{equ:extension}
1 \to N \overset{\iota}{\to} E \overset{\pi}{\to} G \to 1.
\end{align}
We say that two group extensions
$1 \to N \overset{\iota_1}{\to} E_1 \overset{\pi_1}{\to} G \to 1$
and
$1 \to N \overset{\iota_2}{\to} E_2 \overset{\pi_2}{\to} G \to 1$
of $G$ by $N$
are \emph{equivalent}, if there is an isomorphism $\Phi \col E_1 \to E_2$ such that the diagram
\[
\begin{tikzcd}
 &  & E_1 \arrow[rd, "\pi_1"] \arrow[dd, "\Phi"] &  &  \\
1 \arrow[r] & N \arrow[ru, "\iota_1"] \arrow[rd, "\iota_2"] &   &  G \arrow[r] & 1 \\
 &  & E_2 \arrow[ru, "\pi_2"] &   & 
\end{tikzcd}
\]
commutes. 
\end{defn}

Any group extension of $G$ by $N$ induces a homomorphism $\psi \col G \to \Out(N)$; see Subsection \ref{subsec:Basic properties Group Extensions}.
Two equivalent extensions of $G$ by $N$ induce the same such map $\psi \col G \to \Out(N)$. 
We denote by $\Ecl(G,N,\psi)$ the set of group extensions of $G$ by $N$ which induce $\psi$ under this equivalence.
If there is no danger of ambiguity we do not label the maps of the short exact sequence i.e.\ we will write $1 \to N \to E \to G \to 1$ instead of (\ref{equ:extension}).

It is well-known that one may fully characterise $\Ecl(G,N,\psi)$ in terms of ordinary group cohomology:
\begin{theorem} \label{thm:classical group extensions}
 Let $G$ and $N$ be groups and let $\psi \col G \to \Out(N)$ be a homomorphism. Furthermore, let $Z = Z(N)$ be the centre of $N$ equipped with the action of $G$ induced by $\psi$. 
Then there is a class $\omega = \omega(G,N,\psi) \in \Hrm^3(G,Z)$, called \emph{obstruction}, such that $\omega = 0$ in  $\Hrm^3(G,Z)$ if and only if
 $\mathcal{E}(G,N,\psi) \not = \emptyset$.
In this case there is a bijection between the sets $\Hrm^2(G,Z)$ and $\mathcal{E}(G,N,\psi)$.
\end{theorem}
Theorem \ref{thm:classical group extensions} may be found in Theorem 6.6 of \cite{brown}, see also \cite{maclane_original}.
Moreover, for a $G$-module $Z$ it is possible to characterise $\Hrm^3(G,Z)$ in terms of these obstructions: 
\begin{theorem} \label{thm:classical obstructions}
For any $G$-module $Z$ and any $\alpha \in \Hrm^3(G, Z)$  there is a group $N$ with $Z = Z(N)$ and a homomorphism $\psi \col G \to \Out(N)$ extending the action of $G$ on $Z$ such that $\alpha = \omega(G, N, \psi)$.
\end{theorem}
Theorem \ref{thm:classical obstructions} may be found in \cite{brown}, Section IV, 6.
In other words, any three dimensional class in ordinary cohomology arises as an obstruction.

The aim of this paper is to derive analogous statements to Theorem \ref{thm:classical group extensions} and Theorem \ref{thm:classical obstructions} involving \emph{bounded} cohomology.
This will use \emph{quasihomomorphisms} as defined and studied by Fujiwara--Kapovich in \cite{fk}. 
 Let $G$ and $H$ be groups. A set-theoretic function $\sigma \col G \to H$ is called \textit{quasihomomorphism} if the set 
 \[
  D(\sigma) = \{\sigma(g) \sigma(h) \sigma(gh)^{-1} | g,h \in G \}
 \]
is finite.
We note that this is not the original definition of \cite{fk} but both definitions are equivalent; see Proposition \ref{prop:equivalent qhm} and Subsection \ref{subsec:quasimorph and quasihomomorph}.
\begin{defn} \label{defn:bounded extension}
We say that an extension $1 \to N \overset{\iota}{\to} E \overset{\pi}{\to} G \to 1$ of $G$ by $N$ is \emph{bounded}, if there is a (set theoretic) section $\sigma \col G \to E$ such that
\begin{itemize}
\item[(i)] $\sigma \col G \to E$ is a quasihomomorphism and
\item[(ii)] the map $\phi_\sigma \col G \to Aut(N)$ induced by $\sigma$ has finite image in $\Aut(N)$. 
\end{itemize}
\end{defn}
Here $\phi_\sigma \col G \to \Aut(N)$ denotes the set-theoretic map $\phi_\sigma \col g \mapsto \phi_\sigma(g)$ with
$$
^{\phi_\sigma(g)} n = \iota^{-1}(\sigma(g) \iota(n) \sigma(g)^{-1}).
$$
We stress that $\phi_\sigma$ is in general not a homomorphism. See Remark \ref{rmk:convention:conjugation and acting by automorphisms} for the notation.
Condition $(ii)$ may seem artificial but is both natural and necessary; see Remark \ref{rmk:why bounded extensions make sense}.
We denote the set of all bounded extensions of a group $G$ by $N$ which induce $\psi$ by $\Ecl_b(G, N, \psi)$ and mention that this is a subset of $\Ecl(G,N,\psi)$.

Analogously to Theorem \ref{thm:classical group extensions} we will characterise the set $\Ecl_b(G, N, \psi) \subset \Ecl(G, N, \psi)$ using \emph{bounded} cohomology.

\begin{theorem} \label{thm:main}
 Let $G$ and $N$ be groups and suppose that $Z = Z(N)$, the centre of $N$, is
equipped with a norm $\| \cdot \|$ such that $(Z, \| \cdot \|)$ has finite balls. 
Furthermore, let $\psi \col G \to \Out(N)$ be a homomorphism with finite image.
 
 There is a class $\omega_b = \omega_b(G,N,\psi) \in \Hrm_b^3(G,Z)$ such that $\omega_b=0$ in  $\Hrm_b^3(G,Z)$ if and only if
 $\mathcal{E}_b(G,N,\psi) \not = \emptyset$ and $c^3(\omega_b) = \omega$ is the obstruction of Theorem \ref{thm:classical group extensions}.
 If $\mathcal{E}_b(G,N,\psi) \not = \emptyset$, 
then the bijection between the sets  $\Hrm^2(G,Z)$ and $\mathcal{E}(G,N,\psi)$ described in Theorem \ref{thm:classical group extensions} restricts to a bijection between 
 $im(c^2) \subset \Hrm^2(G,Z)$  and $\mathcal{E}_b(G,N,\psi) \subset \Ecl(G, N, \psi)$.
\end{theorem}

Here, $c^n \col \Hrm^n_b(G, Z) \to \Hrm^n(G, Z)$ denotes the \emph{comparison map}; see Subsection \ref{subsec:bounded cohomology}.
We say that a normed group or module $(Z, \| \cdot \|)$ has finite balls if for every $K > 0$ the set $\{ z \in Z \mid \| z \| \leq K \}$ is finite. 
Theorem \ref{thm:main} is applied to examples in Subsection \ref{subsec:examples}.

Just as in Theorem \ref{thm:classical obstructions}
we may ask which elements of $\Hrm^3_b(G, Z)$ may be realised by obstructions.
For a $G$-module $Z$ we define the following subset of $\Hrm^3_b(G,Z)$:
\[
 \mathcal{F}(G,Z) := \{ \Phi^*\alpha \in \Hrm^3_b(G,Z) \mid \Phi \col G \to M \text{ is a homomorphism}, M \text{ a finite group}, \alpha \in \Hrm_b^3(M,Z) \}
\]
where $\Phi^* \alpha$ denotes the pullback of $\alpha$ via the homomorphism $\Phi$. As $M$ is finite, $\Hrm^3(M,Z) = \Hrm^3_b(M,Z)$.
Analogously to Theorem \ref{thm:classical obstructions} we will show:
\begin{theorem} \label{theorem:obstructions}
Let $G$ be a group, let $Z$ be a normed $G$-module with finite balls and such that $G$ acts on $Z$ via finitely many automorphisms. Then
\[
\{ \omega_b(G, N, \psi) \in \Hrm^3_b(G, Z) \mid Z = Z(N) \text{ and } \psi \text{ induces the action on } G \} = \mathcal{F}(G,Z)
\]
as subsets of $\Hrm^3_b(G,Z)$.
\end{theorem}

As finite groups are amenable this shows that all such classes in $\Hrm^3_b(G,Z)$ will vanish under a change to real coefficients; see Subsection \ref{subsec:bounded cohomology}.
We prove Theorem \ref{thm:main} and \ref{theorem:obstructions} following the outline of the classical proofs in \cite{brown}.

\subsection{Organisation of the paper}
This paper is organised as follows:
In Section \ref{sec:preliminaries} we recall well-known facts of (bounded) cohomology and quasihomomorphisms.
 In Section \ref{sec:proof of main theorem} we will reformulate the problem of characterising group extensions using \emph{non-abelian cocycles}; see Definition \ref{defn:non-abelian cocycle}. Using this characterisation, we will prove Theorem \ref{thm:main} in Subsection \ref{subsec:proof main thm}.
In Section \ref{sec:examples and obstructions} we prove Theorem \ref{theorem:obstructions} which characterises the set of classes arising as obstructions $\omega_b$. In Section \ref{sec:exmp and general} we give examples to show that the assumptions of Theorem \ref{thm:main} are necessary and discuss generalisations.
The proof of Proposition \ref{prop:equivalent qhm} is postponed to the Appendix in Section \ref{sec:appendix-equivalent def of qhm}.

\section{Preliminaries} \label{sec:preliminaries}

In this section we recall notation and conventions regarding the (outer) automorphisms in Subsection \ref{subsec:aut and out}.
We further recall basic facts on (bounded) cohomology of groups in Subsection \ref{subsec:bounded cohomology} and quasihomomorphisms by Fujiwara--Kapovich in Subsection \ref{subsec:quasimorph and quasihomomorph}.

\subsection{Notation and conventions, $\Aut$ and $\Out$} \label{subsec:aut and out}
Throughout this paper, Roman capitals ($A$, $B$) denote groups, lowercase Roman letters ($a,b$) denote group elements and greek letters ($\alpha, \beta$) denote functions.
We stick to this notation unless it is mathematical convention to do otherwise.
In a group $G$ the identity will be denoted by $1 \in G$ and by $0 \in G$ to stress that $G$ is abelian.
The trivial group will also be denoted by ``$1$''.

Let $N$ be a group and let $\Aut(N)$ be the group of automorphisms of $N$.
Recall that $\Inn(N)$ denotes the group of \emph{inner automorphisms}. This is, the subgroup of $\Aut(N)$ whose elements are induced by conjugations of elements in $N$. There is a map $\phi \col N \to \Inn(N)$ 
via $\phi \col n \to \phi_n$
 where $\phi_n\col g \mapsto n g n^{-1}$.
Recall that $\Inn(N)$ is a normal subgroup of $\Aut(N)$ and that the quotient $\Out(N) = \Aut(N) / \Inn(N)$ is the \emph{group of outer automorphisms of $N$}.
It is well-known that there is an exact sequence
\[
1 \to Z \to N \to \Inn(N) \to \Aut(N) \to \Out(N) \to 1
\]
where $Z = Z(N)$ denotes the centre of $N$ and all the maps are the obvious ones. 
We will frequently use the following facts. 
Let $G$ be a group.
Any homomorphism $\psi \col G \to \Out(N)$ induces an action on $Z = Z(N)$. This fact is also proved in detail in Subsection \ref{subsec:Basic properties Group Extensions}.
Moreover, if $n_1, n_2 \in N$ are two elements such that for every $g \in N$,
$\phi_{n_1} (g) = \phi_{n_2} (g)$ then $n_1$ and $n_2$ just differ by an element in the centre, i.e.\ there is $z \in Z(N)$ such that $n_1 = z n_2$. This may be seen by the exactness of the above sequence.

\subsection{(Bounded) cohomology of groups} \label{subsec:bounded cohomology}

For what follows we will define (bounded) cohomology using the \emph{inhomogeneous} resolution.
Let $G$ be a group and let $V$ be a $\Z G$-module. In what follows we may refer to a $\Z G$-module simply as $G$-module.
Following \cite{frigerio}, a \emph{norm} on a $G$-module $V$ is a map $\| \cdot \| \col V \to \R^+$ such that
\begin{itemize}
\item $\| v \| = 0$ if and only if $v = 0$
\item $\| r v \| \leq | r | \| v \|$ for every $r \in \Z$, $v \in V$
\item $\| v + w \| \leq \| v \| + \| w \|$
\item $\| g v \| = \| v \|$ for every $g \in G$, $v \in V$.
\end{itemize}
Suppose that the $G$-module $V$ is equipped with a norm $\| \cdot \|$.
Set $C^0(G, V) = C^0_b(G, V) = V$ and set for $n \geq 1$, $C^n(G, V) = \{ \alpha \col G^n \to V \}$.
For an element $\alpha \in C^n(G,V)$ we define
$\| \alpha \| = \sup_{(g_1, \ldots, g_n) \in G^n} \| \alpha(g_1, \ldots, g_n) \|$ when the supremum exists and set $\| \alpha \| = \infty$, else.
For $n \geq 1$ set $C^n_b(G, V) = \{ \alpha \in C^n(G,V) \mid \| \alpha \| < \infty \}$, the \emph{bounded chains}.

We define $\delta^n \col C^n(G,V) \to C^{n+1}(G,V)$, the \emph{coboundary operator}, as follows:
Set $\delta^0 \col C^0(G,V) \to C^1(G,V)$ via $\delta^0(v) \col g_1 \mapsto g_1 \cdot v - v$
and for $n \geq 1$ define $\delta^n \col C^n(G,V) \to C^{n+1}(G,V)$ via
\begin{equ*}{rCl}
 \delta^n(\alpha) \col (g_1, \ldots ,g_{n+1}) &\mapsto &g_1 \cdot \alpha(g_2,\ldots,g_{n+1}) \\
 & &+\sum_{i=1}^n(-1)^i \alpha(g_1,\ldots,g_i g_{i+1}, \ldots, g_{n+1}) \\
 & &+(-1)^{n+1} \alpha(g_1,\ldots, g_n).
\end{equ*}
Note that $\delta^n$ restricts to a map $C_b^n(G,V) \to C_b^{n+1}(G,V)$ for any $n \geq 0$. By abuse of notation we also call this restriction $\delta^n$ as well.

It is well-known that $(C^*(G,V),\delta^*)$ is a cochain complex. The \emph{cohomology of $G$ with coefficients in $V$} is the homology of this complex and denoted by $\Hrm^*(G,V)$.
Similarly $(C_b^*(G,V),\delta^*)$ 
is a cochain complex and its homology is the \emph{bounded cohomology of $G$ with coefficients in $V$} and denoted by $\Hrm^*_b(G,V)$.

Let $W$ be a normed $H$-module and let $\Phi \col G \to H$ be a homomorphism. Denote by $V$ the normed abelian group $W$ equipped with $G$-module structure induced by $\Phi$.
We then obtain a map $\Phi^* \col \Hrm^*(H, W) \to \Hrm^*(G,V)$ via $\Phi^* \col \alpha \mapsto \Phi^* \alpha$ where $\Phi^* \alpha$ denotes the pullback of $\alpha$ via $\Phi$. Similarly we obtain a map
$\Phi^* \col \Hrm_b^*(H, W) \to \Hrm_b^*(G,V)$.

For what follows it will be helpful to work with \emph{non-degenerate} chains. A map $\alpha \in C^n(G,V)$ is called non-degenerate if $\alpha(g_1, \ldots, g_n) = 0$ whenever $g_i = 1$ for some $i=1, \ldots, n$.
We define $NC^0(G,V) = NC^0_b(G,V) = V$ and moreover $NC^n(G, V) = \{ \alpha \in C^n(G, V) \mid \alpha \text{ non-degenerate} \}$ and $NC^n_b(G,V) = \{ \alpha \in C_b^n(G, V) \mid \alpha \text{ non-degenerate} \}$ and observe that $\delta^*$ sends non-degenerate maps to non-degenerate maps.

\begin{prop} \label{prop:non-degenerate cocycles}
The homology of $(NC^*(G,V), \delta^*)$ is $\Hrm^n(G,V)$ and the homology of $(NC^*_b(G,V), \delta^*)$ is $\Hrm_b^n(G,V)$.
\end{prop}

\begin{proof}
See Section 6 of \cite{maclane}, where an explicit homotopy between the complexes $(NC^*(G,V), \delta^*)$ and $(C^*(G,V), \delta^*)$ is constructed. Moreover, one may see that this homotopy preserves bounded maps and hence yields a homotopy between $(NC_b^*(G,V), \delta^*)$ and $(C_b^*(G,V), \delta^*)$.
\end{proof}

Note that the inclusion $C^*_b(G,V) \hookrightarrow C^*(G,V)$ commutes with the coboundary operator and hence induces a well defined map $c^* \col \Hrm^*_b(G,V) \to \Hrm^*(G,V)$, called the \textit{comparison map}.
For a thorough treatment of ordinary and bounded cohomology see \cite{brown} and \cite{frigerio} respectively.

\subsection{Quasimorphisms and quasihomomorphisms} \label{subsec:quasimorph and quasihomomorph}

Consider $\ker(c^2)$ the kernel of the comparison map in dimension $2$. An element $[\omega] \in \ker(c^2)$ is a class of a bounded function $\omega \in C^2_b(G, V)$ which vanishes in ordinary cohomology, i.e.\ such that there is a map, $\phi \col G \to V$ such that for all $g_1,g_2 \in G$, $\omega(g_1,g_2) = \delta^1 \phi(g_1,g_2) = g_1 \cdot \phi(g_2) - \phi(g_1 g_1) + \phi(g_2)$.

If $V = \Z$  is equipped with the trivial $G$-module structure and the standard norm on $\Z$ then we see that every class $[\omega] \in \ker(c^2)$ may be represented by the coboundary of a map $\phi \col G \to \Z$ such that there is a $D > 0$ such that for all $g, h \in G$ $|\phi(g) + \phi(h) - \phi(gh)| \leq D$.
We call such maps \emph{quasimorphisms}.
Maps $\psi \col G \to \Z$ such that there is a constant $D' > 0$ and a homomorphisms $\eta \col G \to \R$ such that for all $g \in G$, $| \phi(g) - \eta(g) | \leq D'$ are called \emph{trivial quasimorphisms}.
Many classes of ``non-positively curved groups'' support non-trivial quasimorphisms. For example acylindrically hyperbolic groups support non-trivial quasimorphisms; see \cite{hull_osin}. On the other hand, amenable groups do not support non-trivial quasimorphisms.
 For a thorough treatment of quasimorphisms see \cite{calegari}.

There are different proposals of how to generalise quasimorphisms $\phi \col G \to \Z$ to maps with an arbitrary group as a target $\phi \col G \to H$.
This paper exclusively treats the generalisation of Fujiwara--Kapovich (\cite{fk}). However, we note that there are other generalisations,
for example one by Hartnick--Schweitzer (\cite{hs}). The latter are considerably more general than the one we are concerned with; see Subsection \ref{subsec:generalisations}.

\begin{defn} \label{def:fk-qm}
 (Fujiwara--Kapovich \cite{fk}) Let $G$ and $H$ be groups and let $\sigma \col G \to H$ be a set-theoretic map. Define $\drm \col G \times G \to H$ via $\drm \col (g,h) \mapsto \sigma(g) \sigma(h) \sigma(gh)^{-1}$ and
define $D(\sigma) \subset H$, the \emph{defect of $\sigma$} via
  \[
  D(\sigma) = \{\drm(g,h) | g,h \in G \} = \{ \sigma(g) \sigma(h) \sigma(gh)^{-1} \mid g,h \in G \}.
 \]
The group 
 $\Delta(\sigma) < H$ generated by $D(\sigma)$ is called the \textit{defect group}.
The map $\sigma \col G \to H$ is called \textit{quasihomomorphism} if the defect $D(\sigma) \subset H$ is finite.
\end{defn}

When there is no danger of ambiguity we will write $D=D(\sigma)$ and $\Delta = \Delta(\sigma)$.
This definition is slightly different from the original definition in \cite{fk}.
Here, the authors required that the set
\[
\bar{D}(\sigma) = \{\sigma(h)^{-1} \sigma(g)^{-1} \sigma(gh) \mid g,h \in G \}
\]
is finite. However, those two definitions may be seen to be equivalent:
\begin{prop} \label{prop:equivalent qhm}
Let $G, H$ be groups and let $\sigma \col G \to H$ be a set-theoretic map. Then $\sigma$ is a quasihomomorphism in the sense of Definition \ref{def:fk-qm} if and only if it is a quasihomomorphism in the sense of Fujiwara--Kapovich (\cite{fk}) i.e.\ if and only if $\bar{D}(\sigma)$ is finite.
\end{prop}
\begin{proof}
We postpone the proof to the Appendix; see Section \ref{sec:appendix-equivalent def of qhm}.
\end{proof}
We use Definition \ref{def:fk-qm} as it is more natural in the context of group extensions.
Every set theoretic map $\sigma \col G \to H$ with finite image and every homomorphism are quasihomomorphisms for ``trivial'' reasons.
We may also construct different quasihomomorphisms using quasimorphisms $\phi \col G \to \Z$:
Let $C < H$ be an infinite cyclic subgroup and let $\tau \col \Z \to H$ be a homomorphism s.t. $\tau(\Z) = C$. Then it is easy to check that for every quasimorphism $\phi \col G \to \Z$, $\tau \circ \phi \col G \to H$ is a quasihomomorphism.

Fujiwara--Kapovich showed that if the target $H$ is a torsion-free hyperbolic group then the above mentioned maps are the only possible quasihomomorphisms.
To be precise in this case every quasihomomorphism $\sigma \col G \to H$ has either finite image, is a homomorphism, or maps to a cyclic subgroup of $H$; see Theorem 4.1 of \cite{fk}.

We recall basic properties of quasihomomorphisms.
For what follows we use the following convention.
\begin{rmk} \label{rmk:convention:conjugation and acting by automorphisms}
If $\alpha \in \Aut(G)$ and $g \in G$ then $^\alpha g$ denotes the element $\alpha(g) \in G$.
If $a \in G$ is an element then $^a g$ denotes conjugation by $a$, i.e.\ the element $a g a^{-1} \in G$. Sometimes we successively apply automorphisms and conjugations. For example, $^a \text{} ^{\alpha} g$ denotes the element $a \alpha(g) a^{-1} \in G$.
\end{rmk}

\begin{prop} \label{prop:basic properties of quasihomomorphisms}
 Let $\sigma \col G \to H$ be a quasihomomorphism, let $D$ and $\Delta$ be as above and let $H_0<H$ be the subgroup of $H$ generated by $\sigma(G)$. Then $\Delta$ is normal in  $H_0$. The function $\phi \col G \to \Aut(\Delta)$
 defined via $\phi(g) \col a \mapsto ^{\sigma(g)} a$ has finite image and its quotient $\psi \col G \to \Out(\Delta)$ is a homomorphism with finite image. Moreover, the pair $(\drm,\phi)$ satisfies
 \[
 ^{\phi(g)} \drm(h,i) \drm(g,hi) = \drm(g,h) \drm(gh,i)
\]
for all $g,h,i \in G$.
 \end{prop} 
Proposition \ref{prop:basic properties of quasihomomorphisms} may be found in Lemma 2.5 of \cite{fk}.
 
 \begin{proof}
For any $g,h,i \in G$ we calculate
\begin{align*}
 \drm(g,h) \drm(gh,i) &= \sigma(g) \sigma(h) \sigma(i) \sigma(ghi)^{-1}  = \sigma(g)\drm(h,i) \sigma(g)^{-1} \drm(g,hi) \\
 &= ^{\phi(g)}\drm(h,i) \drm(g,hi)
\end{align*}
 so $(\drm,\phi)$ satisfies the identity of the proposition.
Rearranging terms we see that
 \[
 ^{\sigma(g)} \drm(h,i) =  \drm(g,h) \drm(gh,i) \drm(g,hi)^{-1}
 \] so $\sigma(g)$ conjugates any $\drm(h,i) \in D$ 
 into the finite set $D \cdot D \cdot D^{-1}$. Here, for two sets $A,B \subset H$, we write $A \cdot B = \{a \cdot b \in H \mid a \in A, b \in B \}$ and $A^{-1}$ denotes the set of inverses of $A$. This shows that $\Delta$ is a normal subgroup of $H_0$, as $D$ generates $\Delta$, and that $\phi \col G \to \Aut(\Delta)$ has finite image. 

To see that the induced map $\psi \col G \to \Out(\Delta)$ is a homomorphism, let
$g,h \in G$ and $a \in \Delta$. Observe that
\begin{align*}
^{\phi(g) \phi(h)} a &= \sigma(g) \sigma(h) a \sigma(h)^{-1} \sigma(g)^{-1} \\
&= ^{\drm(g,h)} \text{}  ^{\sigma(gh)} a
\end{align*}
 and hence $\phi(g) \circ \phi(h)$ and $\phi(gh)$ differ by an inner automorphism. We conclude that $\psi(g) \circ \psi(h) = \psi(gh)$ as elements in $\Out(\Delta)$. So $\psi \col G \to \Out(\Delta)$ is a homomorphism. This shows Proposition \ref{prop:basic properties of quasihomomorphisms}.
 \end{proof}

\begin{rmk} \label{rmk:why bounded extensions make sense}
In light of Proposition \ref{prop:basic properties of quasihomomorphisms} the extra assumption in Theorem \ref{thm:main} that the conjugation by the quasihomomorphism induces a finite image in $\Aut(N)$ is natural: 
Given a short exact sequence $1 \to N \to E \to G \to 1$ that admits a quasihomomorphic section $\sigma: G \to E$ one may see that $1 \to \Delta \to E_0 \to G \to 1$ is a short exact sequence where 
$\Delta = \Delta(\sigma) < N$ and $E_0 = \langle \sigma(G) \rangle < E$ and the map to $\Aut(\Delta)$ has finite image. In fact this assumption is necessary as Example \ref{exmp:aut has finite image necessary} shows.
\end{rmk}

\begin{prop} \label{prop:may assume that sigma one is one}
Let $\sigma \col G \to H$ be a quasihomomorphism. Then the map $\tilde{\sigma} \col G \to H$ defined via
\[
\tilde{\sigma}(g) = \begin{cases}
1 & \text{if } g = 1 \\
\sigma(g) & \text{else}
\end{cases}
\]
is also a quasihomomorphism.
\end{prop}

\begin{proof}
An immediate calculation shows that $D(\tilde{\sigma}) \subset D(\sigma) \cup \{ 1 \}$.
\end{proof}
We will use the last proposition to assume that quasihomomorphic sections of extensions satisfy $\sigma(1) = 1$.

\section{Extensions and proof of Theorem \ref{thm:main}} \label{sec:proof of main theorem}

Recall from the introduction that an extension of a group $G$ by a group $N$ is a short exact sequence
\[
1 \to N \to E \to G \to 1
\]
and that each such extension induces a homomorphism $\psi \col G \to \Out(N)$.
We will recall the construction of such $\psi$ in Subsection \ref{subsec:Basic properties Group Extensions}. 

In Subsection \ref{subsec:group extensions induce non-abelian cocycle} we will define \emph{non-abelian cocycles} (see Definition \ref{defn:non-abelian cocycle}) for group extensions of $G$ by $N$ which induce $\psi$.
Those are certain pairs of functions $(\erm, \phi)$ where $\erm \col G \times G \to N$ and $\phi \col G \to \Aut(N)$.

We will see that every group extension of $G$ by $N$ inducing $\psi$ gives rise to a non-abelian cocycle $(\erm, \phi)$ in Proposition \ref{prop: extension induce non-abelian cocycle}. On the other hand every non-abelian cocycle $(\erm, \phi)$ gives rise to an extension $1 \to N \to  \Erm(\erm, \phi) \to G \to 1$; see Proposition \ref{prop:non-abelian cocycles induce extensions}.
We will use this correspondence to prove Theorem \ref{thm:main} in Subsection \ref{subsec:proof main thm}.
The proof will follow the outline of \cite{brown}, Chapter VI, 6.

\subsection{Group extensions} \label{subsec:Basic properties Group Extensions}

Let $1 \to N \overset{\iota}{\to} E \overset{\pi}{\to} G \to 1$ be an extension of $G$ by $N$ and let $\sigma \col G \to E$ be any set-theoretic section of $\pi \col E \to G$. Then $\sigma \col G \to E$ induces a map $\phi_\sigma \col G \to \Aut(N)$ via $\phi_\sigma(g) \col n \mapsto \iota^{-1}(^{\sigma(g)} \iota(n))$. See Remark \ref{rmk:convention:conjugation and acting by automorphisms} for notation. 
Let $\sigma' \col G \to E$ be another section of $\pi$. For every $g \in G$, $\pi \circ \sigma(g) = \pi \circ \sigma'(g)$ hence there is an element $\nu(g) \in N$ such that $\sigma'(g) = \nu(g) \sigma(g)$. Let $\phi_{\sigma'} \col G \to \Aut(N)$ be the induced map to $\Aut(N)$.
We see that for every $n \in N$,
\[
^{\phi_{\sigma'}(g)} n = ^{\nu(g)} \left( ^{\phi_\sigma(g)} n \right)
\]
so $\phi_{\sigma'}(g)$ and $\phi_\sigma(g)$ only differ by an inner automorphism. We conclude that the projection $\psi \col G \to \Out(N)$ of both $\phi_\sigma$ and $\phi_{\sigma'}$ is the same map $\psi \col G \to \Out(N)$. Hence $\psi$ does not depend on the section.

To see that $\psi$ is a homomorphism, let $g,h \in G$. As $\pi(\sigma(g) \sigma(h) \sigma(gh)^{-1}) = 1$, there is an element $\nu(g,h) \in N$ such that $\iota(\nu(g,h)) = \sigma(g) \sigma(h) \sigma(gh)^{-1}$. In particular, for every $n \in N$,
\[
^{\phi_\sigma(g) \circ \phi_\sigma(h)} n = ^{\nu(g,h)} \left( ^{\phi_\sigma(gh)} n \right)
\]
and hence $\phi_\sigma(g) \circ \phi_\sigma(h)$ and $\phi_\sigma(gh)$ only differ by an inner automorphism, so $\psi(g) \circ \psi(h) = \psi(gh)$ and $\psi \col G \to \Out(N)$ is indeed a homomorphism.

If $1 \to N \overset{\iota_1}{\to} E_1 \overset{\pi_1}{\to} G \to 1$ and $1 \to N \overset{\iota_2}{\to} E_2 \overset{\pi_2}{\to} G \to 1$ are two equivalent group extensions (see Definition \ref{defn:equivalence of group extensions}) with isomorphism $\Phi \col E_1 \to E_2$
and if $\sigma_1 \col G \to E_1$ is a section of $\pi_1 \col G \to E_1$ then it is easy to see that $\sigma_2 = \Phi \circ \sigma_1 \col G \to E_2$ is a section of $\pi_2 \col E_2 \to G$ and that $\phi_{\sigma_1} = \phi_{\sigma_2}$. Hence the induced homomorphism $\psi \col G \to \Out(N)$ is the same.
We collect these facts in a proposition:
\begin{prop} \label{prop:very basic properties group extensions}
Let $1 \to N \overset{\iota}{\to} E \overset{\pi}{\to} G \to 1$ be a group extension of $G$ by $N$. Any two sections $\sigma, \sigma' \col G \to E$ of $\pi$ induce the same homomorphism $\psi \col G \to \Out(N)$. Moreover, two equvivalent group extensions (see Definition \ref{defn:equivalence of group extensions}) induce the same homomorphism $\psi \col G \to \Out(N)$.
\end{prop}

\subsection{Non-abelian cocyclces} \label{subsec:group extensions induce non-abelian cocycle} 

To show Theorem \ref{thm:main} we will transform the problem of finding all group extensions of $G$ by $N$ which induce $\psi$
to the problem of finding certain pairs $(\erm, \phi)$ called \emph{non-abelian cocycles} where $e \col G \times G \to N$ and $\phi \col G \to \Aut(N)$ are certain set-theoretic functions.

\begin{defn} \label{defn:non-abelian cocycle}
Let $G, N$ be groups and let  $\psi \col G \to \Out(N)$ be a homomorphism.
Let $\erm: G \times G \to N$ and $\phi \col G \to \Aut(N)$ be set-theoretic functions such that
\begin{itemize}
\item[(i)] 
 $\phi \col G \to \Aut(N)$ projects to $\psi \col G \to \Out(N)$, $\phi(1) = 1$ and for all $g \in G$, $\erm(1,g) = \erm(g,1) = 1$,
\item[(ii)] for all $g,h \in G$ and $n \in N$, $^{\erm(g,h)} n = ^{\phi(g) \phi(h) \phi(gh)^{-1}} n$ and
\item[(iii)] for all $g,h,i \in G$,  $^{\phi(g)}\erm(h,i) \erm(g,hi) = \erm(g,h) \erm(gh,i)$.
\end{itemize}
Then we say that $(\erm,\phi)$ is a \emph{non-abelian cocycle with respect to $(G, N, \psi)$}.
\end{defn} 
The idea of studying extensions using these non-abelian cocycles is classical; see Chapter IV, 5.6 of \cite{brown}. Here, the author simply calls this a ``cocycle condition''. In order not to confuse it with the cocycle condition of an ordinary $2$-cycle we call it ``non-abelian cocycle'' with respect to the data for group extensions.
Consider Remark \ref{rmk:convention:conjugation and acting by automorphisms} for the notation of conjugation and action of automorphisms.

Every group extension $1 \to N \overset{\iota}{\to} E \overset{\pi}{\to} G \to 1$ that induces $\psi \col G \to \Out(N)$ yields a non-abelian cocycle with respect to $(G,N,\psi)$: As in Subsection \ref{subsec:Basic properties Group Extensions}, pick a set-theoretic section
$\sigma \col G \to E$ such that $\sigma(1) = 1$, define $\phi_\sigma \col G \to \Aut(N)$ via $^{\phi_\sigma(g)} n =   \iota^{-1} \left( ^{\sigma(g)} \iota(n) \right)$ and define $\erm_\sigma \col G \times G \to N$
via $\erm_\sigma \col (g,h) \mapsto \iota^{-1}(\sigma(g) \sigma(h) \sigma(gh)^{-1})$.
Observe that $\sigma$ is a quasihomomorphism if and only if $\erm_\sigma$ has finite image.
\begin{prop} \label{prop: extension induce non-abelian cocycle} 
Let $1 \to N \to E \to G \to 1$ be an extension which induces $\psi$.
\begin{enumerate}
\item \label{item:non-abelian cocycle} For any section $\sigma \col G \to E$ with $\sigma(1) = 1$ the pair $(\erm_\sigma, \phi_\sigma)$ is indeed a non-abelian cocycle with respect to $(G, N, \psi)$.
\item \label{item:general lift} Let $\phi \col G \to \Aut(N)$ be a lift of $\psi$ with $\phi(1) = 1$. Then there is a section $\sigma \col G \to E$ with $\sigma(1) = 1$ such that $\phi_\sigma = \phi$, for $\phi_\sigma$ as above.
If the extension is in addition bounded 
(see Definition \ref{defn:bounded extension}) and $\phi$ has finite image, then $\sigma$ may be chosen to be a quasihomomorphism with $\sigma(1) = 1$.
\end{enumerate}
\end{prop}

\begin{proof}
Part $(1)$ is classical and may be found in the proof of Theorem 5.4 of \cite{brown}.

To see (\ref{item:general lift}), let $\tau \col G \to E$ be any section of $\pi \col E \to G$ with $\tau(1) = 1$. Both $\phi$ and $\phi_\tau$ are lifts of $\psi$ and hence differ only by an inner automorphism.
Let $\nu \col G \to N$ be a representative of such an inner automorphism with $\nu(1)=1$. Then for every $n \in N$, $g \in G$,
\[
^{\phi(g)} n = ^{\nu(g)} \left( ^{\phi_\tau(g)} n \right) = ^{\left( \nu(g) \tau(g) \right) } n.
\]
Let $\sigma \col G \to E$ be the section defined via $\sigma(g) = \nu(g) \tau(g)$. Then we see that $\phi = \phi_\sigma$.
Assume now that the extension is  in addition bounded and that $\phi$ has finite image.
Since the extension is bounded, there is a section $\tau \col G \to E$ which is a quasihomomorphism and such that $\phi_\tau \col G \to \Aut(N)$ has finite image.
By Proposition \ref{prop:may assume that sigma one is one}
we may assume that $\tau(1) = 1$.
We see that we may choose $\nu \col G \to N$ to also have finite image.

We claim that the section $\sigma \col G \to E$ defined via $\sigma \col g \mapsto \nu(g) \tau(g)$ is a quasihomomorphism. Indeed for any $g,h \in G$ we calculate
\begin{align*}
\sigma(g) \sigma(h) \sigma(gh)^{-1} &= 
\nu(g) \tau(g) \nu(h) \tau(h) \tau(gh)^{-1} \nu(gh)^{-1} \\
&= \nu(g) ^{\tau(g)} \nu(h) \left( \tau(g) \tau(h) \tau(gh)^{-1} \right) \nu(gh)^{-1} \\
&\in  \mathcal{N} \mathcal{M} D(\tau) \mathcal{N}^{-1} 
\end{align*}
where $\mathcal{N} = \{ \nu(g) \mid g \in G \}$, the image of $\nu$, $\mathcal{M} = \{ ^{\tau(g)} \nu(h) \mid g,h \}$ which is finite. So all sets on the right hand side are finite and hence $\sigma$ is a quasihomomorphism. This concludes the proof of Proposition \ref{prop: extension induce non-abelian cocycle}.
\end{proof}

\subsection{Non-abelian cocycles yield group extensions} \label{subsec:non-abelian cocycle yield group extensions}
Let $(\erm, \phi)$ be a non-abelian cocycle with respect to $(G,N,\psi)$.
We now describe how $(\erm, \phi)$ gives rise to a group extension $1 \to N \to \Erm(\erm,\phi) \to G \to 1$ which induces $\psi$.
For this we define a group structure on the set $N \times G$ via
\[
(n_1, g_1) \cdot (n_2, g_2) = (n_1 \text{ }^{\phi(g_1)}n_2 \erm(g_1,g_2), g_1 g_2)
\]
for two elements $(n_1, g_1), (n_2, g_2) \in N \times G$.
We denote this group by $\Erm(\erm,\phi)$ and define the maps 
$\iota \col N \to \Erm(\erm, \phi)$ via $\iota \col n \mapsto (n,1)$, $\pi \col \Erm(\erm, \phi) \to G$ via $\pi \col (n,g) \mapsto g$ and $\sigma \col G \to \Erm(\erm, \phi)$ via $\sigma \col g \mapsto (1, g)$.

 \begin{prop} \label{prop:non-abelian cocycles induce extensions} 
 Let $(\erm, \phi )$ be a non-abelian cocycle with respect to $(G, N , \psi)$ and let $\Erm(\erm,\phi)$, $\iota \col N \to \Erm(\erm,\phi)$, $\pi \col \Erm(\erm,\phi) \to G$ and $\sigma \col G \to \Erm(\erm,\phi)$ be as above.
 Then
\begin{enumerate}
\item $1 \to N \overset{\iota}{\to} \Erm(\erm, \phi) \overset{\pi}{\to} G \to 1$
is an extension of $G$ by $N$ inducing $\psi \col G \to \Out(N)$. Moreover, $\sigma$ is a section of $\pi$ such that $\erm = \erm_\sigma$ and $\phi = \phi_\sigma$.
\item If both $\phi \col G \to \Aut(N)$ and $\erm \col G \times G \to N$ have finite image then the extension we obtain is bounded (see Definition \ref{defn:bounded extension}).
\end{enumerate}

 \end{prop}

\begin{proof}

Part $(1)$ is classical; see Chapter IV.6 of \cite{brown} where such extensions from non-abelian cocycles are implicitly constructed.

For part $(2)$, suppose that both $\erm$ and $\phi$ have finite image then the section $\sigma \col G \to \Erm(\erm, \phi)$ is a quasihomomorphism as the defect is just the image of $\erm$ and, moreover, the map $\phi_\sigma = \phi$ has finite image. Hence the extension is bounded. This concludes the proof of Proposition \ref{prop:non-abelian cocycles induce extensions}.
\end{proof}

 For the proof of Theorem \ref{thm:main} we will need to determine when two non-abelian cocycles correspond up to equivalence to the same group extension. We will need the following statement which is stated, though not proved, at the end of IV.6 in \cite{brown}.
 \begin{prop} \label{prop: non-abelian cocycles and cohomology stuff} 
 Let $G$, $N$ be groups, let $\psi \col G \to \Out(N)$ be a homomorphism and let $\phi \col G \to \Aut(N)$ be a lift with $\phi(1) = 1$.
Let $\erm, \erm' \col G \times G \to N$ be two set-theoretic functions such that for all $g \in G$, $\erm(1,g) = \erm(g,1) = 1$ and $\erm'(1,g) = \erm'(g,1) = 1$.
\begin{enumerate}
\item \label{item:two non-abelian cocycles differ by cocycle}
If $(\erm, \phi)$ is a non-abelian cocycle with respect to $(G,N, \psi)$ then $(\erm', \phi)$ is a non-abelian cocycle with respect to $(G, N, \psi)$ if and only if 
 there is a map $\crm \col G \times G \to Z(N) = Z$ satisfying $\delta^2 \crm = 0$ such that for all $g,h \in G$, $\erm'(g,h) = \crm(g,h) \cdot \erm(g,h)$ and for all $g \in G$, $\crm(1,g) = \crm(g,1) = 1$.
\item \label{item:two equivalent non-abelian cocycles} 
If both $(\erm, \phi)$ and $(\erm', \phi)$ are non-abelian cocycles with respect to $(G,N, \psi)$ then the group extensions 
corresponding to $(\erm, \phi)$ and $(\erm', \phi)$ are equivalent if and only if there is a map $\zrm \col G \to Z = Z(N)$ with $\zrm(1) = 1$ such that $\erm(g,h) = (\delta^1 \zrm)(g,h) \erm'(g,h)$.
\end{enumerate} 
Recall that $Z(N) = Z$ denotes the centre of $N$. 
 \end{prop}
 
\begin{proof}
 To see (\ref{item:two non-abelian cocycles differ by cocycle}), note that for every $g,h \in G$, $n \in N$,
 \[
 ^{\erm(g,h)} n = ^{\phi(g) \phi(h) \phi(gh)^{-1} } n = ^{\erm'(g,h)} n
 \]
 by $(ii)$ of Definition \ref{defn:non-abelian cocycle}. 
 Hence there is an element $\crm(g,h) \in Z(N)$ such that $\erm'(g,h) = \crm(g,h) \erm(g,h)$ and for all $g \in G$, $\crm(1,g) = \crm(g,1) = 1$. Moreover, for every $g,h,i \in G$,
 \begin{align*}
 ^{\phi(g)}\erm'(h,i) \erm'(g,hi) &= \erm'(g,h) \erm'(gh,i) \\
 ^{\phi(g)} \crm(h,i) ^{\phi(g)} \erm(h,i) \crm(g,hi) \erm(g,hi) &= \crm(g,h) \erm(g,h) \crm(gh,i) \erm(gh,i) \\
 (\delta^2 \crm (g,h,i)) ^{\phi(g)}\erm(h,i) \erm(g,hi) &= \erm(g,h) \erm(gh,i) \\
\delta^2 \crm(g,h,i) &= 1
 \end{align*}
and hence for $\delta^2 \crm = 0$ if we restrict to $Z$. On the other hand the same calculation shows that if $(\erm, \phi)$ is a non-abelian cocycle and $\crm \col G \times G \to Z(N)$ satisfies $\delta^2 \crm = 0$ then $(\erm', \phi)$ is a non-abelian cocycle with $\erm'(g,h) = \crm(g,h) \erm(g,h)$.
 
For (\ref{item:two equivalent non-abelian cocycles}) suppose that there is a $\zrm \col G \to Z$ as in the proposition.
Define the map $\Phi \col \Erm(\erm, \phi) \to \Erm(\erm', \phi)$ via $\Phi \col (n,g) \mapsto (n \zrm(g), g)$.
Then for every $(n_1,g_1), (n_2,g_2) \in \Erm(\erm, \phi)$,
\begin{align*}
\Phi \left( (n_1,g_1) \right) \cdot \Phi \left( (n_2,g_2) \right) &= (n_1 \zrm(g_1), g_1) \cdot (n_2 \zrm(g_2),g_2) \\
&= (n_1 ^{\phi(g_1)} n_2 \zrm(g_1) ^{\phi(g_2)} \zrm(g_2) \erm'(g_1, g_2), g_1 g_2) \\ 
&=  (n_1 ^{\phi(g_1)} n_2 \zrm(g_1 g_2) \delta^1 \zrm(g_1,g_2) \erm'(g_1, g_2), g_1 g_2)  \\
&=  (n_1 ^{\phi(g_1)} n_2 \erm(g_1,g_2) \zrm(g_1 g_2) , g_1 g_2  )\\ 
&= \Phi \left( (n_1,g_1) \cdot (n_2,g_2)\right) 
\end{align*}
and hence $\Phi$ is a homomorphism. It is easy to see that $\Phi$ is an isomorphism and that $\Phi$ fits into the diagram of Definition \ref{defn:equivalence of group extensions}. Hence the extensions corresponding to $(\erm, \phi)$ and $(\erm', \phi)$ are equivalent.

On the other hand suppose that the extensions 
$1 \to N \overset{\iota}{\to} \Erm(\erm, \phi) \overset{\pi}{\to} G \to 1$ and $1 \to N \overset{\iota'}{\to} \Erm(\erm', \phi) \overset{\pi'}{\to} G \to 1$ are equivalent with sections $\sigma, \sigma'$ as before with Isomorphism $\Phi \col \Erm(\erm, \phi) \to \Erm(\erm', \phi)$.

Note that for all $g \in G$, $\pi' \circ \Phi \left( (1,g) \right) = g$ and hence 
the second coordinate of $\Phi((1,g)) \in \Erm(\erm, \phi)$ is $g$. Define $\zrm \col G \to N$ via $\Phi((1,g)) = (\zrm(g),g)$.
Observe that $^{\sigma(g)} \iota(n) = (^{\phi(g)} n, 1)$ and $^{\sigma'(g)} \iota(n) =    (^{\phi(g)} n,1)$ and hence $\sigma(g)$ and $\sigma'(g)$ only differ by an element in the centre hence $\zrm(g) \in Z$.
Note that for every $g,h \in G$,
\begin{align*}
(\erm(g,h),1) &= \sigma(g) \sigma(h) \sigma(gh)^{-1} \\
\Phi \Big( (\erm(g,h),1) \Big) &= \Phi \Big( \sigma(g) \Big) \cdot \Phi \Big( \sigma(h) \Big) \cdot  \Phi \Big( \sigma(gh) \Big)^{-1} \\
(\erm(g,h),1) &= (\zrm(g) ^{\phi(g)} \zrm(h) \zrm(gh)^{-1} \erm'(g,h), 1).
\end{align*}
Comparing the last line we see that 
$\erm(g,h) = \delta^1 \zrm(g,h) \erm'(g,h)$ which concludes the proposition. 
\end{proof}

\subsection{Proof of Theorem \ref{thm:main}} \label{subsec:proof main thm}
We can now prove Theorem \ref{thm:main} using the correspondence of group extensions with non-abelian cocycles.

\begin{reptheorem}{thm:main}
 Let $G$ and $N$ be groups and suppose that $Z = Z(N)$, the centre of $N$, is
equipped with a norm $\| \cdot \|$ such that $(Z, \| \cdot \|)$ has finite balls. 
Furthermore, let $\psi \col G \to \Out(N)$ be a homomorphism with finite image.
 
 There is a class $\omega_b = \omega_b(G,N,\psi) \in \Hrm_b^3(G,Z)$ such that $\omega_b=0$ in  $\Hrm_b^3(G,Z)$ if and only if
 $\mathcal{E}_b(G,N,\psi) \not = \emptyset$ and $c^3(\omega_b) = \omega$ is the obstruction of Theorem \ref{thm:classical group extensions}.
 If $\mathcal{E}_b(G,N,\psi) \not = \emptyset$, 
then the bijection between the sets  $\Hrm^2(G,Z)$ and $\mathcal{E}(G,N,\psi)$ described in Theorem \ref{thm:classical group extensions} restricts to a bijection between 
 $im(c^2) \subset \Hrm^2(G,Z)$  and $\mathcal{E}_b(G,N,\psi) \subset \Ecl(G, N, \psi)$.
\end{reptheorem}

Recall that a normed $G$-module $Z$ is said to have finite balls if for every $K > 0$ the set $\{ z \in Z \mid \| z \|  \leq K \}$ is finite.
We will split the proof into several claims.
Claim \ref{claim:existence of zeta to N} associates to a tuple $(G,N, \psi)$ as in the theorem a function $\zeta \col G \times G \to N$ which we then use to define the obstruction class $ \omega_b = [\orm_b] \in \Hrm_b^3(G,Z)$ in Equation (\ref{equ:obstruction}).   
In Claims \ref{claim:o maps to centre and is cocycle} and 
\ref{claim:obstruction independent of choices}
we see that $\orm_b$ is indeed a bounded cocycle and that $\omega_b = [\orm_b] \in \Hrm_b^3(G,Z)$ is independent of the choices made.
Finally in Claim \ref{claim:obstruction encodes extensions exist} we see that $\omega_b$ indeed encodes if (bounded) extensions for the data $(G,N, \psi)$ exist.
In Claim \ref{claim:correspondence h2 and extensions} we construct a bijection $\Psi$ between $\Hrm^2(G,Z)$ (resp. $\im(c^2)$) and (bounded) extensions.

Let $G$, $N$, $\psi \col G \to \Out(N)$ and $Z, \| \cdot \|$ be  as in the theorem.
Choose a lift $\phi \col G \to \Aut(N)$ of $\psi$ with finite image such that $\phi(1) = 1$.

\begin{claim} \label{claim:existence of zeta to N}
There is a function $\zeta \col G \times G \to N$ such that for all $g,h \in G$, $n \in N$,
\[
^{\zeta(g,h)} n = ^{\phi(g) \phi(h) \phi(gh)^{-1}} n
\]
where $\zeta$ has finite image in $N$ and for all $g \in G$, $\zeta(g,1)=\zeta(1,g)=1$.
\end{claim}

\begin{proof}[Proof of Claim \ref{claim:existence of zeta to N}]
For $g,h \in G$ we have that $\psi(g) \psi(h) \psi(gh)^{-1} = 1$, since $\psi$ is a homomorphism. Hence for every $g,h \in G$, the map $\phi(g) \phi(h) \phi(gh)^{-1} \in \Aut(N)$ is an inner automorphism.

As $\phi$ has finite image in $\Aut(N)$, the function $(g,h) \mapsto \phi(g) \phi(h) \phi(gh)^{-1}$ has finite image in $\Inn(N) < \Aut(N)$.
We may find a lift $\zeta \col G \times G \to N$ of this map such that $\zeta$ has finite image and such that $\zeta(1,g) = \zeta(g,1) = 1$. This shows Claim \ref{claim:existence of zeta to N}.
\end{proof}

We now define the obstruction class. 
Define $\orm_b \col G \times G \times G \to N$ so that for all $g,h,i \in G$, 
\begin{equation} \label{equ:obstruction}
^{\phi(g)} \zeta(h,i) \zeta(g,hi) = \orm_b(g,h,i) \zeta(g,h) \zeta(gh,i)
\end{equation}
and observe that $\orm_b$ necessarily has finite image as both $\zeta \col G \times G \to N$ and $\phi \col G \to \Aut(N)$ have finite image.
Also, observe that $\orm_b(g,h,i) = 1$ if one of $g,h,i \in G$ is trivial.

\begin{claim} \label{claim:o maps to centre and is cocycle}
 The function $\orm_b \col G \times G \times G \to N$ maps to $Z=Z(N)<N$ the centre of $N$. 
 Moreover, $\orm_b$ is a non-degenerate bounded cocycle, i.e.\ $\delta^3 \orm_b = 0$.
 \end{claim}

\begin{proof}[Proof of Claim \ref{claim:o maps to centre and is cocycle}]
First we show that $\orm_b$ maps to the centre of $N$.
Observe that for all $g,h,i \in G$ and $n \in N$,
\begin{align*}
^{^{\phi(g)} \zeta(h,i) \zeta(g,hi)} n &= ^{ \phi(g) \phi(h) \phi(i) \phi(hi)^{-1} \phi(g)^{-1} } (^{\phi(g) \phi(hi) \phi(ghi)^{-1} } n ) \\
&= ^{\phi(g) \phi(h) \phi(i) \phi(ghi)^{-1}} n \\
&= ^{\phi(g) \phi(h) \phi(gh)^{-1}} ( ^{\phi(gh) \phi(i) \phi(ghi)^{-1} }  n ) \\
&= ^{\zeta(g,h) \zeta(gh,i)} n
\end{align*}
and hence $^{\phi(g)} \zeta(h,i) \zeta(g,hi)$ and $\zeta(g,h) \zeta(gh,i)$ induce the same map by conjugation on $N$ and hence just differ by an element of the centre so $\orm_b(g,h,i) \in Z$. 
Since $\zeta$ and $\phi$ have finite image, so does $\orm_b$, i.e.\ $\orm_b \in C^3_b(G,Z)$ and it is easy to see that $\orm_b$ is non-degenerate.

To see that $\orm_b$ satisfies $\delta^3 \orm_b = 0$ we calculate 
\begin{align*} 
 ^{\phi(g) \phi(h)} \zeta(i,k) ^{\phi(g)} \zeta(h,ik) \zeta(g,hik)
\end{align*}
for $g,h,i,k \in G$ in two different ways. First observe that
\begin{align*}
 ^{\phi(g) \phi(h)} \zeta(i,k) ^{\phi(g)} \zeta(h,ik) \zeta(g,hik) = & ^{\phi(g) \phi(h)} \zeta(i,k) \left( ^{\phi(g)} \zeta(h,ik) \zeta(g,hik) \right) \\
 = & ^{\phi(g) \phi(h)} \zeta(i,k) \orm_b(g,h,ik) \zeta(g,h) \zeta(gh,ik) \\
 = & \zeta(g,h) ^{\phi(gh)} \zeta(i,k) \orm_b(g,h,ik) \orm_b(g,h,ik) \\
 = & \zeta(g,h) \zeta(gh,i) \zeta(ghi,k)  \orm_b(g,h,ik) \orm_b(gh,i,k)
\end{align*}
then observe that
\begin{align*}
 ^{\phi(g) \phi(h)} \zeta(i,k) ^{\phi(g)} \zeta(h,ik) \zeta(g,hik) = & \left( ^{\phi(g) \phi(h)} \zeta (i,k)  ^{\phi(g)} \zeta (h,ik) \right) \zeta (g,hik)  \\
= & ^{\phi(g)} \left( \orm_b(h,i,k) \zeta(h,i) \zeta(hi,k) \right) \zeta(g,hik) \\
= &  ^{\phi(g)} \zeta(h,i)  \zeta(g,hi) \zeta(ghi,k) ^{\phi(g)} \orm_b(h,i,k) \orm_b(g,hi,k) \\
= & \zeta(g,h) \zeta(gh,i) \zeta(ghi,k) \orm_b(g,h,i)  ^{\phi(g)} \orm_b(h,i,k) \orm_b(g,hi,k).
\end{align*}
Finally, comparing these two terms yields
\[
 \delta^3 \orm_b(g,h,i,k) = ^{\phi(g)} \orm_b(h,i,k) - \orm_b(gh,i,k) + \orm_b(g,hi,k) - \orm_b(g,h,ik) + \orm_b(g,h,i) = 0.
\]
So $\orm_b$ indeed defines a bounded cocycle. This shows Claim \ref{claim:o maps to centre and is cocycle}.
\end{proof}

 \begin{claim} \label{claim:obstruction independent of choices}
 The class $[\orm_b] \in \Hrm^3_b(G,Z)$ is independent of the choices made for $\zeta$ and $\phi$.
\end{claim}

\begin{proof}[Proof of Claim \ref{claim:obstruction independent of choices}]
Let $\phi, \phi' \col G \rightarrow \Aut(N)$ be two lifts of $\psi$ as above and choose corresponding functions $\zeta$, $\zeta' \col G \to N$ representing the defect of $\phi$ and $\phi'$ as above.
There is a finite function $\nu \col G \rightarrow N$ with finite image such that $\phi(g) = \bar \nu(g) \phi'(g)$ where $\bar \nu(g)$ is the element in $Inn(N) \subset \Aut(N)$ corresponding to the conjugation by $\nu(g)$.
We calculate
\begin{align*}
\phi(g) \phi(h) \phi(gh)^{-1} &= \bar \nu(g) ^{\phi'(g)} \bar \nu(h) \left( \phi'(g) \phi'(h) \phi'(gh)^{-1} \right) \bar \nu(gh)^{-1}.
\end{align*}
We see that for every $n \in N$,
\begin{align*}
^{\zeta(g,h)} n &= ^{\phi(g) \phi(h) \phi(gh)^{-1}} n \\
&= ^{\bar \nu(g) ^{\phi'(g)} \bar \nu(h) \left( \phi'(g) \phi'(h) \phi'(gh)^{-1} \right) \bar \nu(gh)^{-1} } n \\
&= ^{\nu(g) ^{\phi'(g)} \nu(h) \zeta'(g,h) \nu(gh)^{-1}} n.
\end{align*}
So $\zeta(g,h)$ and $\nu(g) ^{\phi'(g)} \nu(h) \zeta'(g,h) \nu(gh)^{-1}$ only differ by an element of the centre.
Hence define $z(g,h) \in Z$ via
\[
 \zeta(g,h) = z(g,h) \nu(g) ^{\phi'(g)} \nu(h) \zeta'(g,h) \nu(gh)^{-1}
\]
and note that $z:G \times G \rightarrow Z$ is a function with finite image
as all functions involved in its definition have finite image.

It is a calculation to show that $\orm_b$, the obstruction defined via the choices $\phi$ and $\zeta$ and $\orm'_b$, the obstruction defined via the choices $\phi'$ and $\zeta'$ differ by $\delta^2 z$ and hence define the same class in bounded cohomology. This shows Claim \ref{claim:obstruction independent of choices}.
\end{proof}

 We call this class $[\orm_b] \in \Hrm^3_b(G,Z)$ the \emph{obstruction for extensions $G$ by $N$ inducing $\psi$}
 and denote it by $\omega_b(G,N, \psi)$ or $\omega_b$. We have seen that $\omega_b$ is a well defined class that depends only on $G$, $N$ and $\psi \col G \to \Out(N)$. Next we show that it is an obstruction to (bounded) extensions.
 
\begin{claim} \label{claim:obstruction encodes extensions exist}
Let $\omega_b \in \Hrm^3_b(G,Z)$ be as above. Then $\omega_b = 0 \in \Hrm^3_b(G,Z)$ if and only if $\Ecl_b(G,N,\psi) \not = \emptyset$. Moreover, $c^3(\omega_b)$ is equal to the classical obstruction.
\end{claim} 
Recall that $c^3 \col \Hrm^3_b(G, Z) \to \Hrm^3(G,Z)$ denotes the comparison map.

 \begin{proof}[Proof of Claim \ref{claim:obstruction encodes extensions exist}]
Suppose that $c^3(\omega_b) = 0 \in \Hrm^3(G,Z)$. Then there is $\beta \in C^2(G,Z)$ possibly with unbounded, i.e.\ infinite image, such that
\begin{align} \label{equ:cocycle}
\orm_b(g,h,i) = ^{\phi(g)} \beta(h,i) -\beta(gh,i) + \beta(g,hi) - \beta(g,h)
\end{align}
for all $g,h,i \in G$. Moreover we may choose $\beta$ such that for all $g \in G$, $\beta(1,g) = \beta(g,1) = 0$ by Proposition \ref{prop:non-degenerate cocycles} since $\orm_b$ is non-degenerate.

 Define $\erm \col G \times G \to N$ via  $\erm(g,h) = \zeta(g,h) \beta(g,h)^{-1}$.
We will show that $(\erm, \phi)$ is a non-abelian cocycle with respect to $(G, N, \psi)$. Indeed, $\phi$ is a lift of $\psi$ which satisfies $\phi(1) = 1$ and for all $g \in G$, $\erm(g,1) = \erm(1,g) = 1$. Moreover, observe that for all $g,h \in G$ and $n \in N$,
\[
^{\erm(g,h)} n = ^{\zeta(g,h) \beta(g,h)^{-1}} n = ^{\zeta(g,h)} n = ^{ \phi(g) \phi(h) \phi(gh)^{-1}} n
\]
as $\beta(g,h)$ is in the centre of $N$.
Finally, for all $g,h,i \in G$ we calculate
\begin{align*}
^{\phi(g)} \zeta(h,i) \zeta(g,hi) &= \orm_b(g,h,i) \zeta(g,h) \zeta(gh,i) \\
^{\phi(g)} \left( \zeta(h,i) \beta(h,i)^{-1} \right) \zeta(g,hi) \beta(g,hi)^{-1} &= \zeta(g,h) \beta(g,h)^{-1} \zeta(gh,i) \beta(gh,i)^{-1} \\
^{\phi(g)}\erm(h,i) \erm(g,hi) &= \erm(g,h) \erm(gh,i)
\end{align*}
and hence indeed $(\erm, \phi)$ is a non-abelian cocycle with respect to $(G,N,\psi)$.

 By Proposition \ref{prop:non-abelian cocycles induce extensions}, $(\erm, \phi)$ gives rise to an extension of $G$ by $N$ which induces $\psi$ and hence $\Ecl(G,N,\psi) \not = \emptyset$.

Analogously, suppose that $\omega_b = 0$ in $\Hrm^3_b(G, Z)$.
Then we may find $\beta \in C^3_b(G, Z)$ satisfying Equation (\ref{equ:cocycle}), but with bounded i.e.\ \emph{finite} image.
Hence if we set $\erm(g,h) = \zeta(g,h) \beta(g,h)^{-1}$, we see that $\erm(g,h)$ has finite image as well, as both $\zeta$ and $\erm$ have.
By the above argument $(\erm, \phi)$ is a non-abelian cocycle and, as both $\erm$ and $\phi$ have finite image, $(\erm,\phi)$ gives rise to a bounded extension of $(N, G, \psi)$ by $(2)$ of Proposition \ref{prop:non-abelian cocycles induce extensions}.
Hence $\Ecl_b(G,N,\psi) \not = \emptyset$.

On the other hand, suppose that $\Ecl(G,N,\psi) \not = \emptyset$. 
This means that there is some extension $1 \to N \to E \to G \to 1$ of $G$ by $N$ which induces $\psi$.
By Propositin \ref{prop: extension induce non-abelian cocycle}, there is a section $\sigma \col E \to G$ such that 
$\phi_\sigma = \phi$ and then $(\erm_\sigma, \phi)$ is a non-abelian cocycle with respect to $(G, N, \psi)$.

Observe that for all $g,h \in G$, $n \in N$,
\[
^{\erm_\sigma(g,h)} n = ^{\phi(g) \phi(h) \phi(gh)^{-1}} n = ^{\zeta(g,h)} n
\]
and hence there is an $\beta(g,h) \in Z<N$ such that $\erm_\sigma(g,h) = \zeta(g,h) \beta(g,h)^{-1}$. 
As $(\erm_\sigma, \phi)$ satisfies $(iii)$ of Definition \ref{defn:non-abelian cocycle}, we see that for all $g,h,i \in G$
\begin{align*}
^{\phi(g)}(\erm_\sigma(h,i)) \erm_\sigma(g,hi) &= \erm_\sigma(g,h) \erm_\sigma(gh,i) \\
^{\phi(g)} \left( \zeta(h,i) \beta(g,h)^{-1} \right) \zeta(g,hi) \beta(g,hi)^{-1} &= \zeta(g,h) \beta(g,h)^{-1} \zeta(gh,i) \beta(gh,i)^{-1} \\
^{\phi(g)} \zeta(h,i) \zeta(g,hi) &= \left( ^{\phi(g)} \beta(h,i) -\beta(gh,i) + \beta(g,hi) - \beta(g,h) \right) \zeta(g,h) \zeta(gh,i)
\end{align*}
so
\[
\orm_b(g,h,i) = ^{\phi(g)} \beta(h,i) -\beta(gh,i) + \beta(g,hi) - \beta(g,h) = \delta^2 \beta (g,h,i)
\]
and hence $c^3(\omega_b) = 0 \in \Hrm^3(G, Z)$.

Now suppose that $\Ecl_b(G,N,\psi) \not = \emptyset$. 
This means that there is some extension $1 \to N \to E \to G \to 1$ of $G$ by $N$ which induces $\psi$ and which is in addition bounded. Applying $(2)$ of Proposition \ref{prop: extension induce non-abelian cocycle} once more we see that there is a section $\sigma \col G \to E$ such that $\sigma$ is a quasihomomorphism satisfying that $\sigma(1) = 1$ by Proposition \ref{prop:may assume that sigma one is one} and $\phi_\sigma = \phi$.
As $\sigma$ is a quasihomomorphism, $\erm_\sigma$ has finite image.

As $\erm_\sigma$ and $\zeta$ have finite image the map $\beta \in C^2(G, Z)$ defined via
$\erm_\sigma(g,h) = \zeta(g,h) \beta(g,h)^{-1}$ also has finite image and hence $\beta \in C^2_b(G, Z)$.
The above calculations show that $\orm_b = \delta^2 \beta$ and hence $\omega_b = 0$ in $\Hrm^3_b(G, Z)$.
This finishes the proof of Claim \ref{claim:obstruction encodes extensions exist}.
 \end{proof}

Now suppose that $\Ecl_b(G, N, \psi) \not = \emptyset$.
then there is an extension $1 \to N \to E_0 \to G \to 1$ which induces $\psi$ 
and a section $\sigma_0 \col G \to E_0$ such that $\phi = \phi_{\sigma_0}$ and $\erm_0 :=\erm_{\sigma_0}$ have finite image and $(\erm_0, \phi)$ is a non-abelian cocycle with respect to $(G, N, \psi)$.

\begin{claim} \label{claim:correspondence h2 and extensions}
Let $\Psi \col \Hrm^2(G, Z) \to \Ecl(G, N, \psi)$ be the map defined via 
\[
\Psi \col [\alpha] \mapsto \left( 1 \to N \to \Erm(\alpha \cdot \erm_0, \phi) \to G \to 1 \right),
\]
where $\alpha$ is a non-degenerate representative. Then $\Psi$ 
is a bijection which restricts to a bijection between $\im(c^2) \subset \Hrm^2(G,Z)$ and $\Ecl_b(G, N, \psi) \subset \Ecl(G, N, \psi)$.
\end{claim}
Here $\alpha \cdot \erm_0$ denotes the map $\alpha \cdot \erm_0 \col G \times G \to N$ defined via $\alpha \cdot \erm_0 \col (g,h) \mapsto \alpha(g,h) \cdot \erm_0(g,h)$. 

\begin{proof}[Proof of Claim \ref{claim:correspondence h2 and extensions}]
We first show that the above map is 
well defined: Let $\alpha \in C^2(G, Z)$ be a non-degenerate cocycle. Then $\delta^2 \alpha = 0$ and hence by Proposition \ref{prop: non-abelian cocycles and cohomology stuff}, $(\alpha \cdot \erm_0, \phi)$ is a non-abelian cocycle with respect to $(G,N, \psi)$.
If $[\alpha'] = [\alpha]$ in $\Hrm^2(G,Z)$ then there is an element $\zrm \in C^1(G, Z)$ such that $\alpha = \alpha' + \delta^1 \zrm$. Then, according to point $(2)$ of Proposition \ref{prop: non-abelian cocycles and cohomology stuff}, the group extensions are equivalent. 
Hence $\Psi$ is well defined.

Now suppose that $\Psi([\alpha]) = \Psi([\alpha'])$.
Then, according to Proposition \ref{prop: non-abelian cocycles and cohomology stuff} ($2$) we have that there is a $\zrm \in C^1(G, Z)$ such that
$(\delta^1 \zrm) \alpha' \erm_0 = \alpha \erm_0$ and hence
$\delta^1 \zrm \alpha' = \alpha$. Hence $[\alpha] = [\alpha']$ in $\Hrm^2(G,Z)$, so $\Psi$ is injective.

Next we show that $\Psi$ is surjective.
Let $1 \to N \to E' \to G \to 1$ be any extension of $G$ by $N$ inducing $\psi$.
By Proposition \ref{prop: extension induce non-abelian cocycle}, there is a section $\sigma' \col G \to E$ such that $\phi_{\sigma'} = \phi$ and such that $(\erm', \phi)$ is a non-abelian cocycle with $\erm' = \erm_{\sigma'}$.
Hence both $(\erm', \phi)$ and $(\erm_0, \phi)$ are non-abelian cocycles with respect to $(G, N, \psi)$ and by Proposition \ref{prop: non-abelian cocycles and cohomology stuff} there is a map $\beta \in C^2(G, Z)$ such that
$\erm' = \beta \cdot \erm_0$ and $\delta^2 \beta = 0$. Then $\beta$ induces a class and hence $\Psi([\beta])$ corresponds to this extension. This shows that $\Psi$ is surjective and hence that $\Psi$ is a bijection.
If $1 \to N \to E' \to G \to 1$ is a bounded extension then we may choose a section $\sigma' \col G \to E'$ such that $\erm'$ as above has finite image. Moreover, $\beta$ as above is bounded as both $\erm'$ and $\erm_0$ are. Hence $[\beta] \in \im(c^2)$ and hence $\Phi(\im(c^2)) \supset \Ecl_b(G,N,\psi)$.

Suppose that $[\alpha] \in \im(c^2)$. Then we may assume that $\alpha \in C^2_b(G, Z)$, i.e.\ that $\alpha$ has finite image and that $\alpha$ is non-degenerate.
Hence $\alpha \cdot \erm_0$ has finite image and hene the extension corresponding to $(\alpha \cdot \erm_0, \phi)$ is bounded by $(2)$ of Proposition \ref{prop:non-abelian cocycles induce extensions}.
This shows that $\Psi(\im(c^2)) \subset \Ecl_b(G,N, \psi)$.
\end{proof}

This concludes the proof of Theorem \ref{thm:main}.

\section{The set of obstructions and examples} \label{sec:examples and obstructions}

Theorem \ref{thm:main} provides a characterisation of non-trivial classes $\omega_b \in \Hrm^3_b(G,Z)$, called obstructions.
One may wonder which such classes $\omega_b \in \Hrm^3_b(G,Z)$ arise in this way.
Recall that in the case of general group extensions, every cocycle in $\Hrm^3(G,Z)$ may be realised as such an obstruction:

\begin{reptheorem}{thm:classical obstructions}
For any $G$-module $Z$ and any $\alpha \in \Hrm^3(G, Z)$  there is a group $N$  with $Z = Z(N)$ and a homomorphism $\psi \col G \to \Out(N)$ extensing the $G$-action on $Z$ such that $\alpha = \omega(G, N, \psi)$ in $\Hrm^3(G,N, \psi)$.
\end{reptheorem}

For a normed $G$-module $Z$ with finite balls
and a $G$-action with finite image define the \emph{set of bounded obstructions} $\mathcal{O}_b(G,Z) \subset \Hrm_b^3(G,Z)$ as
\[
 \mathcal{O}_b(G,Z) = \{ \omega_b(G,N,\psi) \in \Hrm^3_b(G,Z) \mid Z = Z(N), ^{\psi(g)}z = g \cdot z, \psi: G \to \Out(N) \text{ finite} \}.
\]
We refer to the introduction for the definition of $\mathcal{F}(G,Z)$ and observe that Theorem \ref{theorem:obstructions} from the introduction may now be restated as follows:
\begin{reptheorem}{theorem:obstructions}
Let $G$ be a group and $Z$ be a normed $G$-module with finite balls and a $G$-action with finite image. Then
\[
\mathcal{O}_b(G,Z) = \mathcal{F}(G,Z)
\]
as subsets of $\Hrm^3_b(G,Z)$.
\end{reptheorem}

This fully characterises obstructions we obtain in bounded cohomology.
\begin{proof}
We have just seen that $\mathcal{O}_b(G,Z) \subset \mathcal{F}(G,Z)$, as we may choose $\omega_b$ in the proof of Theorem \ref{thm:main} so that it factors through $\Out(N)$ via $\psi \col G \to \Out(N)$ and $\Out(N)$ is a finite group.

To show $\mathcal{F}(G,Z) \subset \mathcal{O}_b(G,Z)$ we need to show that for every finite group $M$ and any class $\alpha \in \Hrm^3(M,Z)$ there is a group $N$ and a homomorphism $\psi \col M \to \Out(N)$
which induces $\alpha$ as a cocycle.
We recall a construction from \cite{maclane}. Working with non-degenerate cocycles (see Subsection \ref{subsec:bounded cohomology}) we may assume that $\alpha(1,g,h) = \alpha(g,1,h) = \alpha(g,h,1) = 0$ for all $g,h \in G$.

Define the abstract symbols $\langle g,h \rangle$ for each $1 \not = g,h \in M$ and set $\langle g,1 \rangle= \langle 1,g \rangle = \langle 1,1 \rangle=1$ for the abstract symbol $1$.  Let $F$ be the free group on these symbols and set $1$ to be the identity element and set $N = Z \times F$. Define the function $\phi: M \rightarrow \Aut(N)$ so
that for $g \in M$  the action of $\phi(g)$ on $F$ is given by
\[
^{\phi(g)} \langle h,i \rangle  = \alpha(g,h,i) \langle g,h \rangle \langle gh,i \rangle \langle g,hi \rangle^{-1}
\]
and so that the action of $\phi(g)$ on $Z$ is given by the $M$-action on $Z$.
A direct calculation yields that for each $g \in M$, the map $\phi(g) \col N \to N$ indeed defines an isomorphism. Here, we need the assumption $\alpha(1,g,h)=\alpha(g,1,h)=\alpha(g,h,1)=0$.
It can be seen that for all $n \in N$ and $g_1, g_2 \in F$
\[
^{\phi(g_1) \phi(g_2)} n = ^{\langle g_1,g_2 \rangle} \text{} ^{\phi(g_1 g_2)} n
\]
where we have to use the fact the $\alpha$ is a cocycle.
Hence, $\phi \col M \to \Aut(N)$ is well defined and induces a \emph{homomorphism} $\psi \col M \to \Out(N)$. It is easy to see that $\psi$ induces the $M$-action on $Z$.
If $M \not \cong \Z_2$, the centre of $N$ is $Z$. In this case, to calculate $\omega_b(M,N,\psi)$ we choose as representatives for $\phi(g)\phi(h)\phi(gh)^{-1}$ simply $\langle g,h \rangle$ and then see by definition that $\omega_b(M,N,\psi)$ is precisely $\alpha$. 

If $M = \Z_2$ then the centre of $N$ is not $Z$.
However, we can enlarge $M$ by setting $\tilde{M} = M \times \Z_2$.
We have both a homomorphism $\pi: \tilde{M} \to M$ via $(m,z) \mapsto m$ and a homomorphism $\iota: M \to \tilde{M}$ via $m \mapsto (m,1)$ such that $\pi \circ \iota = id_M$. Let $\tilde \alpha \in \Hrm^3(\tilde{M},Z)$ be the pullback of $\alpha$ via $\pi$. Let $\tilde{N}$ be the group constructed as above with this cocycle and let $\tilde \phi: \tilde{M} \to \Aut(\tilde N)$ and $\tilde \psi: \tilde{M} \to \Out(\tilde N)$ be the corresponding functions.
The centre of $\tilde{N}$ is $Z$. 
Set $\psi \col M \to \Out(\tilde{N})$ via $\psi = \tilde \psi \circ \iota$.
Then the obstruction $\omega_b(M, \tilde{N}, \psi)$ can be seen to be $\alpha$. This shows Theorem \ref{theorem:obstructions}.
\end{proof}

\section{Examples and Generalisations} \label{sec:exmp and general}

We discuss Examples in Subsection \ref{subsec:examples} where we show in particular that the requirements in Definition \ref{defn:bounded extension} are necessary. Subsection \ref{subsec:generalisations} discusses possible generalisations of Theorem \ref{thm:main}.

\subsection{Examples} \label{subsec:examples}
The subset $\mathcal{E}_b(G,N,\psi) \subset \mathcal{E}(G,N,\psi)$ is generally neither empty nor all of $\mathcal{E}(G,N,\psi)$. For any hyperbolic group we have $\mathcal{E}_b(G,N,\psi) = \mathcal{E}(G,N,\psi)$ as the comparison map is surjective (\cite{mineyev}).
 We give different examples where the inclusion $\mathcal{E}_b(G,N,\psi) \subset \mathcal{E}(G,N,\psi)$ is strict.

 The examples we discuss will use the \emph{Heisenberg group} $\Heis$. This group fits into the central extension
 \[
1 \to \Z \to \Heis \to \Z^2 \to 1.
 \]
Elements of the Heisenberg group will be described by $[c,z]$, where $c \in \Z$ and $z \in \Z^2$. The group multiplication is given by
$
[c_1,z_1] \cdot [c_2, z_2] = [c_1 + c_2 + \omega(z_1, z_2), z_1 + z_2]
$
where $\omega(z_1,z_2) = \det(z_1,z_2)$, the determinant of the $2\times 2$-matrix $(z_1,z_2)$. Observe that $[c,z]^{-1} = [-c,-z]$, and that $^{[c_1,z_1]} [c_2,z_2] = [c_2 + 2 \omega(z_1, z_2), z_2]$.
The inner automorphisms are isomorphic to $\Z^2$ with the identification $\phi \col \Z^2 \to \Inn(\Heis)$ via $^{\phi(g)} [c,z] = [c + 2 \omega(g,z), z]$.
It is well-known that $\omega$ generates $\Hrm^2(\Z^2, \Z)$ and that $\omega$ can not be represented by a bounded cocycle, i.e.\ the comparison map $c^2 \col \Hrm^2_b( \Z^2, \Z) \to \Hrm^2( \Z^2, \Z)$ is trivial. 
 
\begin{exmp}  \label{exmp:first example heisenberg}
Let $G = \Z^2$, $N = \Z$ and let $\psi \col G \to \Out(N)$ be the homomorphism with trivial image. The direct product
\begin{align} \label{equ:trivial heisenberg}
1 \to N \to N \times G \to G \to 1
\end{align}
\sloppy clearly has a quasihomomorphic section that induces a finite map to $\Aut(N)$ and hence
$\mathcal{E}_b(G, N, \psi)  \not =  \emptyset$.
Let $Z(N) = \Z$ be equipped with the standard norm. Note that $\im(c^2) = \{ 0 \}$, for  $c^2 \col \Hrm^2_b(\Z^2, \Z) \to \Hrm^2(\Z^2, \Z)$ the comparison map. By Theorem \ref{thm:main}, $\Ecl_b(\Z^2, \Z, \psi)$ 
consists of exactly one element, which is the direct product described above. Note that the Heisenberg extension
 \[
 1 \to \Z \to \Heis \to \Z^2 \to 1
 \]
is not equivalent to (\ref{equ:trivial heisenberg}). This can be seen as $\Heis$ is not abelian.
Hence this extension is not bounded.
So in this case 
\[
\emptyset \not = \mathcal{E}_b(\Z^2,\Z,id) \subsetneq \mathcal{E}(\Z^2,\Z,id).
\]
\end{exmp}

\begin{exmp} \label{exmp:aut has finite image necessary}
The assumption that the quasihomomorphism $\sigma \col G \to E$ has to induce a map $\phi_\sigma \col G \to \Aut(N)$ with finite image may seem artificial, as the induced homomorphism $\psi \col G \to \Out(N)$ already has finite image.
However it is necessary as the following example shows.

Consider extensions of $G = \Z^2$ by $N = \Heis$ which induce $\psi \col G \to \Out(N)$ with trivial image. Again, $\Ecl_b(G, N, \psi)$ is not empty as it contains the extension corresponding to the direct product $1 \to \Heis \to \Heis \times \Z^2 \to \Z^2 \to 1$.
Moreover, $Z(N) = Z(\Heis) = \Z$ and just as in Example \ref{exmp:first example heisenberg} the comparison map
$c^2 \col \Hrm^2_b(\Z^2, \Z) \to \Hrm^2(\Z^2, \Z)$
is trivial, i.e.\ $\im(c^2) = \{ 0 \}$.
So up to equivalence there is just one bounded extension, namely the one corresponding to the direct product $\Heis \times \Z^2$.

Pick an isomorphism $\phi \col \Z^2 \to \Inn(\Heis)$.
We may construct the extension
\begin{align} \label{equ:z2 by heisenberg extension}
1 \to \Heis \to \Heis \rtimes_{\phi} \Z^2 \to \Z^2 \to 1
\end{align}
where $\Heis \rtimes_{\phi} \Z^2$ denotes the semi-direct product.
and observe that the action of $\Z^2$ on the centre of $\Heis$ is trivial as the automorphisms are all inner.
\begin{claim} \label{claim:heisenberg not equivl to direct }
The extension (\ref{equ:z2 by heisenberg extension}) is not equivalent to the extension $1 \to \Heis \to \Heis \times \Z^2 \to \Z^2 \to 1$. 
\end{claim}

\begin{proof}
Indeed, we show that $\Heis \rtimes_{\phi} \Z^2$ is not isomorphic to $\Heis \times \Z^2$.
We will show that $\Zrm(\Heis \times \Z^2) \cong \Z^3$ and $\Zrm(\Heis \rtimes_{\phi} \Z^2) \cong \Z$, where $\Zrm(G)$ denotes the centre of the group $G$.
First, observe that 
$$
\Zrm(\Heis \times \Z^2) \cong \Zrm(\Heis) \times \Zrm(Z^2) \cong \Z \times \Z^2 \cong \Z^3.
$$
Now assume that $([c,z],n), ([c',z'],n') \in \Heis \rtimes_{\phi} \Z^2$ are two elements which commute. Then
\begin{eqnarray*}
([c,z],n) \cdot ([c',z'],n') &=& ([c',z'],n') \cdot ([c,z],n) \\
([c,z] \cdot ^n [c',z'], n + n') &=& ([c',z'] \cdot ^{n'} [c,z], n + n') \\
([c,z] \cdot [c' + 2 \omega(n,z'),z'], n + n') &=& ([c',z'] \cdot  [c + 2 \omega(n',z),z], n + n') \\
([c+ c' + 2 \omega(n,z') + \omega(z,z'), z + z'], n + n') &=& ([c'+ c + 2 \omega(n',z) + \omega(z',z) ,z + z'], n + n')
\end{eqnarray*}
and hence such elements satisfy
$$
\omega(n,z') =  \omega(n',z) + \omega(z',z) = \omega(n'+z',z).
$$
Hence, if $([c,z],n)$ is in the centre of
$\Heis \rtimes_{\phi} \Z^2$, then
$n$ and $z$ must be such that the above equation holds for every choice of $n'$ and $z'$, and hence $n=z=0 \in \Z^2$.
We conclude that the centre of  $\Heis \rtimes_{\phi} \Z^2$ is $\{ ([c,0],0) \mid c \in \Z \}$  and that
$$
\Zrm(\Heis \rtimes_{\phi} \Z^2) \cong \Z.
$$
Hence $\Heis \times \Z^2$ and $\Heis \rtimes_{\phi} \Z^2$ cannot be isomorphic.
\end{proof}

So extension (\ref{equ:z2 by heisenberg extension}) is not bounded.
On the other hand there are two special sort of sections $\sigma \col G \to \Heis \rtimes_{\phi} \Z^2 $:
\begin{itemize}
\item[(i)] The section $\sigma_1 \col g \mapsto (1,g)$ to (\ref{equ:z2 by heisenberg extension}) is a homomorphism and hence in particular a quasihomomorphism.
However, the induced map $\phi_{\sigma_1} \col G \to \Aut(\Heis)$, has as the image the full \emph{infinite} group of inner automorphisms.

\item[(ii)] On the other hand, the section $\sigma_2 \col g \mapsto ([1,-g],g)$ induces a trivial map $\phi_{\sigma_2} \col G \to \Aut(\Heis)$ as seen in the proof of Claim \ref{claim:heisenberg not equivl to direct }.
Indeed we calculate that for $g, h \in G$,
\[
\sigma_2(g) \sigma_2(h) \sigma_2(gh)^{-1} = ([\omega(g,h),0],0)
\]
and so $D(\sigma_2)$ is unbounded and $\sigma_2$ is not a quasihomomorphism.
\end{itemize}

We conclude that there is a section $\sigma_1$ which satisfies $(i)$ of Definition \ref{defn:bounded extension} and another section $\sigma_2$ which satisfies $(ii)$ of Definition \ref{defn:bounded extension} but no section which satisfies $(i)$ and $(ii)$ simultaneously. 
\end{exmp}

\subsection{Generalisations} \label{subsec:generalisations}

One interesting aspect of Theorem \ref{thm:main} is that it characterises certain classes in
\emph{third bounded cohomology}, namely the obstructions. 
Moreover we have seen that the obstructions for \emph{bounded} extensions factor through a finite group.
Finite groups are amenable and hence all such classes in third bounded cohomology will vanish when passing to real coefficients.

On the other hand every class in third \emph{ordinary} cohomology may be realised by an obstruction; see Theorem \ref{thm:classical obstructions}.
One may wonder if there is another type of extensions $\tilde{\Ecl} \subset \Ecl(G,N,\psi)$ which is empty if and only if a certain class $\tilde{\omega}$ is non-trivial in $\Hrm^3_b(G, \R)$. This would be interesting as non-trivial classes in third bounded cohomology with real coefficients are notoriously difficult to construct.

Recall that our Definition \ref{defn:bounded extension} of \textit{bounded} extensions $1 \to N \to E \to G \to 1$ required the existence of sections $\sigma \col G \to E$ which satisfied two conditions. Namely $(i)$ that $\sigma$ is a quasihomomorphism, and $(ii)$ that the induced a map $\phi_\sigma \col G \to \Aut(N)$ by conjugation has finite image.
One may wonder if 
a modification of conditions $(i)$ and $(ii)$ yield different such obstructions with different coefficients.
For modifications of $(i)$ there are some generalisations of the quasimorphisms by Fujiwara--Kapovich, most notably the one by Hartnick--Schweitzer \cite{hs}. However, there does not seem to be a natural generalisation of condition $(ii)$, i.e.\ a generalisation of $\phi_\sigma$ having finite image. However, such a generalisation  is necessary as else the obstructions factor through a finite group and will yield trivial classes with real coefficients. On the other hand, there has to be some restrictions on the sort of sections $\sigma$ allowed:
Consider the bounded cohomology of a free non-abelian group $F$.
Soma \cite{soma} showed that $\Hrm^3_b(F,\R)$ is infinite dimensional. But every extension $1 \to N \to E \to F \to 1$ will even have a homomorphic section $\sigma \col F \to E$. Without a condition on $\phi_\sigma$ there would be no obstruction for such extensions.

\section{Appendix: Equivalent Definitions of Quasihomomorphisms} \label{sec:appendix-equivalent def of qhm}

We now prove Proposition \ref{prop:equivalent qhm} which shows that the definition of quasihomomorphism given in \cite{fk} is equivalent to Definition \ref{def:fk-qm}.
Recall that for a set-theoretic map $\sigma \col G \to H$ we defined $\bar{D}(\sigma) \subset H$ as
\[
\bar{D}(\sigma) := \{ \sigma(h)^{-1} \sigma(g)^{-1} \sigma(gh) \mid g,h \in G \}.
\]
Suppose that $\sigma \col G \to H$ is a quasihomomorphism in the sense of Definition \ref{def:fk-qm}. We start by noting the following easy property.
\begin{claim} \label{claim:property_finite_conj}
Let $\sigma \col G \to H$ be a quasihomomorphism with defect group $\Delta$ and let $A \subset \Delta$ be a finite subset of $\Delta$. Then the set
\[
\{ ^{\sigma(g)} A  \mid g \in G \}
\]
is also a finite subset of $\Delta$.
\end{claim}
\begin{proof}
By Proposition \ref{prop:basic properties of quasihomomorphisms}, the set of automorphisms $\{ a \mapsto ^{\sigma(g)} a\mid g \in G \} \subset \Aut(\Delta)$ is finite.
Hence we see that the set 
$\{ ^{\sigma(g)} A  \mid g \in G \}$ is the image of a finite set of $\Delta$ under finitely many automorphisms of $\Delta$ and hence a finite subset of $\Delta$.
\end{proof}

Recall that $D=D(\sigma)$, the defect of $\sigma$, is defined as $D(\sigma) := \{ \sigma(g) \sigma(h) \sigma(gh)^{-1} \mid g,h \in G \}$.
Observe that $\drm(1,1) = \sigma(1) \sigma(1) \sigma(1)^{-1} = \sigma(1)$ and hence $\sigma(1) \in D$. 
Moreover, we see that $\drm(g,g^{-1}) = \sigma(g) \sigma(g^{-1}) \sigma(1)^{-1}$, hence
$\sigma(g)^{-1} \in \sigma(g^{-1}) \cdot D_0$, where $D_0 = \sigma(1)^{-1} \cdot D^{-1} \subset \Delta$, a finite set.
Combining the above expressions we see that for every $g,h \in G$,
\[
\sigma(h)^{-1} \sigma(g)^{-1} \sigma(gh) \in \sigma(h^{-1}) D_0 \sigma(g^{-1}) D_0 D_0^{-1} \sigma((gh)^{-1})^{-1}.
\]

Now observe that the set
\[
D_1 = \{ \sigma(g^{-1}) D_0 D_0^{-1} \sigma(g^{-1})^{-1} \mid g \in G \} \subset \Delta
\] is finite by Claim \ref{claim:property_finite_conj}.
Hence
\begin{align*}
\sigma(h)^{-1} \sigma(g)^{-1} \sigma(gh) \in \sigma(h^{-1}) D_0 D_1 \sigma(g^{-1}) \sigma((gh)^{-1}).
\end{align*}
Using the claim again we see that
\[
D_2 = \{ \sigma(h^{-1}) D_0 D_1 \sigma(h^{-1})^{-1} \mid h \in G \}
\]
is finite and hence that
\[
\bar{D}(\sigma) = \{ \sigma(h)^{-1} \sigma(g)^{-1} \sigma(gh) \mid g,h \in G \}  \subset D_2 \sigma(h^{-1}) \sigma(g^{-1}) \sigma((gh)^{-1}) \subset D_2 D 
\]
so $\bar{D}(\sigma)$ is indeed a finite set.
This shows that any quasihomomorphism in the sense of Definition \ref{def:fk-qm} is a quasihomomorphism in the sense of \cite{fk}.

Now assume that $\sigma \col G \to H$ is a map such that the set $\bar{D} = \bar{D}(\sigma)$ is finite
and let $\bar{\Delta}$ be the group generated by $\bar{D}$.

Just as before we have the following claim:
\begin{claim} \label{claim:finite set for fk qhm}
Let $f \col G \to H$ be a map such that $\bar{D} = \bar{D}(f)$ is finite and let $\bar{\Delta}$ be the group generated by $\bar{D}$. If $A \subset \bar{\Delta}$ if a finite subset of $\bar{\Delta}$ then the set
\[
\{ ^{\sigma(g)^{-1}} A  \mid g \in G \}
\]
is also a finite subset of $\bar{\Delta}$.
\end{claim} 
\begin{proof}
This follows from the same argument as for Claim \ref{claim:property_finite_conj} using Lemma 2.5 of \cite{fk} instead of Proposition \ref{prop:basic properties of quasihomomorphisms}.
\end{proof}

Observe again that $\sigma(1)^{-1} = \sigma(1)^{-1} \sigma(1)^{-1} \sigma(1) \in \bar{D}(\sigma)$ and using that for all $g \in G$, $\sigma(g)^{-1} \sigma(g^{-1})^{-1} \sigma(1) \in \bar{D}$ we see that
$\sigma(g) \in \sigma(g^{-1})^{-1} \bar{D}_0$ where $\bar{D}_0 = \sigma(1) \bar{D}$.

Hence for every $g,h \in G$,
\[
\sigma(g) \sigma(h) \sigma(gh)^{-1} \in \sigma(g^{-1})^{-1} \bar{D}_0 \sigma(h^{-1})^{-1} \bar{D}_0 \bar{D}_0^{-1} \sigma(h^{-1} g^{-1})
\]
By Claim \ref{claim:finite set for fk qhm}, we see that the set
\[
\bar{D}_1 = \{ \sigma(h^{-1})^{-1} \bar{D}_0 \bar{D}_0^{-1} \sigma(h^{-1}) \mid h \in G \}
\]
is finite and hence
\[
\drm(g,h) = \sigma(g) \sigma(h) \sigma(gh)^{-1} \in \sigma(g^{-1})^{-1} \bar{D}_0 \bar{D}_1 \sigma(h^{-1})^{-1} \sigma(h^{-1} g^{-1}).
\]
Using the claim once more we see that the set
\[
\bar{D}_2 = \{ f(g^{-1})^{-1} \bar{D}_0 \bar{D}_1 f(g^{-1}) \mid g \in G \}
\]
is finite. Finally,
\[
\drm(g,h) = \sigma(g) \sigma(h) \sigma(gh)^{-1} \in \bar{D}_2 \sigma(g^{-1})^{-1} \sigma(h^{-1})^{-1} \sigma(h^{-1} g^{-1}) \subset \bar{D}_2 \bar{D}
\]
which is a finite set. Hence $D(\sigma)$ is finite. So every quasihomomorphism in the sense of \cite{fk} is also a quasihomomorphism in the sense of Definition \ref{def:fk-qm}.

\begin{acknowledgements}
I would like to thank my supervisor Martin Bridson for his helpful comments and support.
I would further like to thank the anonymous reviewer for many very helpful comments.
The author was funded by the Oxford-Cocker Graduate Scholarship. Part of this work was written at the Isaac Newton Institue while participating in the programme
\emph{Non-Positive Curvature Group Actions and Cohomology}, supported by the EPSRC Grant EP/K032208/1.
\end{acknowledgements}

\bibliographystyle{mscplain}
\bibliography{bib_QS}

\end{document}